\title{Beta Expansions for Regular Pisot Numbers}
\author{
       	Maysum Panju\thanks{Research of M. Panju was partially supported by NSERC, the UW President's Research Award, UW Undergraduate Research Internship, and the department of Pure Mathematics at the University of Waterloo.} \\
Department of Pure Mathematics \\
       	University of Waterloo
}
\date{\today}

\documentclass[12pt]{article}

\usepackage{amssymb, amsmath, amsthm}
\usepackage{fullpage}
\usepackage{pdflscape, multirow}
\usepackage{longtable}
\usepackage{graphicx, float}

\newtheorem{theorem}{Theorem}[section]
\newtheorem{lemma}[theorem]{Lemma}

\theoremstyle{definition}
\newtheorem{definition}[theorem]{Definition}

\begin{document}
\maketitle
\begin{abstract} 
A beta expansion is the analogue of the base 10 representation of a real number, where the base may be a non-integer. Although the greedy beta expansion of 1 using a non-integer base is in general infinitely long and non-repeating, it is known that if the base is a Pisot number, then this expansion will always be finite or periodic. Some work has been done to learn more about these expansions, but in general these expansions were not explicitly known. In this paper, we present a complete list of the greedy beta expansions of 1 where the base is any regular Pisot number less than 2, revealing a variety of remarkable patterns. We also answer a conjecture of Boyd's regarding cyclotomic co-factors for greedy expansions.
\end{abstract}

\section{Introduction}

The base 10 numbering system is firmly established as the conventional number representation scheme for common use. Familiar and well-behaved, the system is easily generalized to allow representation schemes that use a base other than 10. In particular, there are several interesting and unusual properties that arise when the base is chosen to be a non-integer. This field of study was first explored by R\'{e}nyi \cite{MR0097374} in 1957, and there have been a number of significant developments in the area since then; many of these advancements have been summarized by Hare \cite{MR2427721}. 

We begin by introducing the concept of a beta expansion, which generalizes the base 10 expansion to allow the non-integer base $\beta$.

\begin{definition}
Let $\beta \in (1, 2)$. Consider the expression 
\[ x = \sum_{i = 1}^\infty a_i \beta^{-i} \]
for some $x \in \left[ 0, 1/\left(\beta - 1\right) \right]$, where $a_i \in \{0, 1\}$. Then $a_1 a_2 a_3 \cdots$ is a \emph{beta expansion} for $x$. 
\end{definition}

In general, the base $\beta$ can be any complex number with $|\beta| > 1$; similarly, the set of usable digits can be more complicated than $\{0,1\}$. These variations, however, are not the focus of this paper. 

Unlike in the base 10 system, there are usually many different valid beta expansions for a fixed $x$ using some non-integer base $\beta$. We will therefore identify with the largest possible expansion. 

\begin{definition}
If $a_1 a_2 a_3 \cdots$ is the (lexicographically) maximal beta expansion for $x$ with base $\beta$, then the expansion $a_1 a_2 a_3 \cdots$ is said to be the \emph{greedy expansion} for $x$ with base $\beta$. 
\end{definition}

These expansions are called ``greedy'' because they can be computed efficiently using a greedy algorithm. For the remainder of this paper, we will concentrate only on greedy beta expansions, and therefore we will use the terms ``greedy expansion'' and ``beta expansion'' interchangeably.

A lot of work involving greedy expansions has been done by Parry \cite{MR0142719}, who provided the following useful characterization theorem.

\begin{theorem}
\label{greed}
Let $a_1 a_2 a_3 \cdots$ be a sequence in $\{0, 1\}^{\mathbb{N}}$. Then
this sequence is the greedy expansion of $1$ for some $\beta > 1$ if and only if for all $j \geq  1$,
\[ a_{j+1} a_{j+2} a_{j+3} \cdots <_\text{lex} a_1 a_2 a_3 \cdots.\]
\end{theorem}
Here, we use $<_{\text{lex}}$ to represent the lexicographic ordering; that is, we say that $a_1 a_2 a_3 \cdots <_\text{lex} b_1 b_2 b_3 \cdots$ if there is some integer $m \geq 1$ with $a_i = b_i$ for all $i < m$, and $a_m < b_m$.

For the remainder of this paper, when looking at beta expansions with base $\beta$, we will restrict our attention to the beta expansion of 1, as opposed to that of general values of $x$. We will use $d_\beta (1)$ to denote the greedy expansion of 1 using base $\beta$. It has been shown that there are specific algebraic properties of $\beta$ that can be recovered if $d_\beta(1)$ is known to terminate or repeat \cite{MR0142719}.

In the case that an expansion of 1 is finite or eventually periodic, the expansion can be encoded in the form of a polynomial, as follows:

\begin{definition}
Consider the greedy expansion $d_\beta(1) = a_1 a_2 \cdots a_n$ if the greedy expansion is finite, and $d_\beta(1) = a_1 a_2 \cdots a_k \left( a_{k+1} \cdots a_n \right)^\omega$ if the greedy expansion is eventually periodic. Define $P_j(x) = x^j - a_1 x^{j-1} - \ldots - a_j$. Then the \emph{companion polynomial} for the expansion is defined as
\[ 
  R(x) =
  \begin{cases}
    P_n(x), & \text{if } d_\beta(1) \text{ is finite;} \\
    P_n(x) - P_k(x), & \text{if } d_\beta(1) \text{ is eventually periodic.} \\
  \end{cases}
\]
\end{definition}

It is straightforward to show, using the definition of the beta expansion and of the companion polynomial, that we have $R(\beta) = 0$. Since $R(x)$ is a monic, integer polynomial, it follows that the numbers $\beta$ for which $d_\beta(1)$ is finite or eventually periodic must be algebraic integers. If the minimal polynomial for $\beta$ is $p(x)$, then we can write $R(x) = p(x)Q(x)$ for some integer polynomial $Q(x)$. We refer to the (possibly constant) polynomial $Q(x)$ as the co-factor for the expansion.

We now turn our attention to a special set of algebraic integers known as Pisot numbers. We will see that beta expansions have a close connection with Pisot numbers.

\begin{definition}
A \emph{Pisot number} is a real algebraic integer $q > 1$ such that all of the conjugates of $q$ have absolute value strictly less than 1. If $P(x)$ is a polynomial with exactly one Pisot number as a zero, that Pisot number is referred to as the \emph{Pisot root of $P(x)$}.
\end{definition}

An example of a Pisot number is the golden ratio, $\tau = (1 + \sqrt{5})/2 \approx 1.618$, which is the Pisot root of $x^2 -x-1$; here, the only conjugate is $-\tau^{-1} \approx -0.618$, which has magnitude less than 1.

It turns out that Pisot numbers behave very nicely with respect to beta expansions. For a general algebraic integer $\beta$, the expansion $d_\beta(1)$ will not necessarily terminate or repeat. However, if the base $\beta$ is a Pisot number, then the greedy expansion of 1 is guaranteed to either be finite or eventually periodic. In fact, if $\beta$ is a Pisot number, then every element of $\mathbb{Q}(\beta)$ has a finite or eventually periodic greedy expansion using $\beta$ as the base \cite{MR0447134, MR576976}. The converse of this is also true.

Although it has long been understood that the Pisot numbers give eventually periodic greedy expansions of 1, the expansions themselves were not explicitly known. In this paper, we present a complete description of these expansions for a large subset of Pisot numbers in the interval $(1, 2)$. However, before we look at the expansions, it is necessary to understand the structure of the set of Pisot numbers in this interval.

It is known that the set of Pisot numbers is closed, and thus contains all of its limit points. Amara \cite{MR0237459} has completely determined all of the limit points in the set of Pisot numbers less than 2.

\begin{theorem}
\label{limits}
The limit points of the Pisot numbers in the interval $(1,2)$ are the following:
\[ \phi_1 = \psi_1 < \phi_2 < \psi_2 < \phi_3 < \chi < \psi_3 < \phi_4 < \cdots < \psi_r < \phi_{r+1} < \cdots < 2 \]
where
\begin{itemize}
\item the minimal polynomial of $\phi_r$ is $\Phi_r(x)= x^{r+1} -2x^r +x -1$,
\item the minimal polynomial of $\psi_r$ is $\Psi_r(x)= x^{r+1} - x^r - \ldots - x -1$, and
\item the minimal polynomial of $\chi$ is $\mathcal{X}(x)= x^4 -x^3 -2x^2 +1$.
\end{itemize}
\end{theorem}

The sequences of Pisot numbers that approach these limit points were first described by Talmoudi \cite{MR516773}, and are also completely understood. The Pisot numbers in these sequences are the roots of the polynomials listed in Table~\ref{regulars}. As $n$ tends to infinity, the Pisot root of each polynomial in the table approaches the corresponding limit point. 

\renewcommand{\arraystretch}{2}
\begin{table}
\caption{Defining Polynomials for Regular Pisot Numbers} \label{regulars}

\begin{center}
\begin{tabular}{ll}
\hline 

\hline
Limit Point & Defining Polynomials \\ 
\hline 
 $\phi_r$ & $\Phi_{A(r,n)}^\pm (x) = \Phi_r(x)x^n \pm \left(x^r -x^{r-1} +1 \right)$ \\ 
 & $\Phi_{B(r,n)}^\pm (x) = \Phi_r(x)x^n \pm \left(x^r -x +1 \right)$ \\ 
 & $\Phi_{C(r,n)}^\pm (x) = \Phi_r(x)x^n \pm \left(x^r+1 \right)\left(x -1 \right)$ \\ 
 $\psi_r$ & $\Psi_{A(r,n)}^\pm (x) = \Psi_r(x)x^n \pm \left(x^{r+1} -1 \right)$ \\ 
 & $\Psi_{B(r,n)}^\pm (x) = \Psi_r(x)x^n \pm \left(x^r -1 \right) / \left(x-1 \right)$ \\ 
 $\chi$ & $\mathcal{X}_{A(n)}^\pm (x) = \mathcal{X}(x)x^n \pm \left( x^3 +x^2 -x-1 \right)$ \\ 
 &  $\mathcal{X}_{B(n)}^\pm (x) = \mathcal{X}(x)x^n \pm \left( x^4 -x^2 +1 \right)$\\
 \hline
 
 \hline 
\end{tabular}
\end{center}
\end{table} 

Note that in general, the polynomials in Table~\ref{regulars} are not irreducible; they may have some cyclotomic factors. However, the polynomials only admit a single zero larger than 1, which is the Pisot number. 

It turns out that the Pisot roots of the polynomials in these sequences are the only Pisot numbers that can get arbitrarily close to one of the limit points mentioned in Theorem~\ref{limits}.
 
\begin{definition}
A \emph{regular Pisot number} in the interval $(1, 2)$ is any Pisot number that is the root of one of the polynomials described in Theorem~\ref{limits} or listed in Table~\ref{regulars}. That is, the regular Pisot numbers are the limit points $\phi_r$, $\psi_r$, and $\chi$, as well as any of the Pisot numbers in the sequences that approach them, as outlined in Table~\ref{regulars}. A Pisot number that is not regular is called \emph{irregular}.
\end{definition}

It is known that for any $\varepsilon > 0$, there are only finitely many irregular Pisot numbers in $[1, 2-\varepsilon]$. Thus, the regular Pisot numbers comprise a significant portion of the Pisot numbers less than 2. Now that we have a full description of what the regular Pisot numbers are, we can look at the beta expansions of 1 associated with these Pisot numbers.

\section{Results}
\label{results}

The main contribution of this paper is the complete description of the expansion $d_\beta(1)$, where the base $\beta$ is any regular Pisot number in the interval $(1,2)$. This work extends the contributions given by Boyd \cite{MR1325863}, Bassino \cite{MR1966122}, Gjini \cite{MR2024974},and Hare and Tweedle \cite{MR2444222}. Using the classifications given by Amara \cite{MR0237459} and Talmoudi \cite{MR516773}, every regular Pisot number in the interval $(1,2)$ is the zero of some polynomial described in Theorem~\ref{limits} or listed in Table~\ref{regulars}. The greedy expansions of 1 using most of these Pisot numbers are relatively straightforward, and are summarized in Section~\ref{simple}. The remaining expansions are more complex in nature, and are discussed individually in Section~\ref{complicated}. Graphical representations for all of the expansions are presented in Section~\ref{graphics}.

Each beta expansion given is accompanied by an associated pseudo co-factor. This is the rational function $Q(x)$ such that $R(x) = P(x)Q(x)$, where $R(x)$ is the companion polynomial of the expansion and $P(x)$ is the (not necessarily minimal) polynomial with the Pisot number as a root, as shown in Theorem~\ref{limits} and Table~\ref{regulars}. In general, $Q(x)$ might not be a polynomial, as it may include factors of $1/g(x)$ for some cyclotomic factor $g(x)$ of $P(x)$. If $P(x)$ is an irreducible polynomial, then the pseudo co-factor for the corresponding expansion would be the same as the ordinary co-factor as described earlier.

\subsection{Simple Expansions}
\label{simple}

The beta expansions of 1 using most of the regular Pisot numbers are summarized below in Table~\ref{expansions}. The expansions are organized by the polynomials which have the regular Pisot numbers as roots. The expansions also depend on the values of $r$ and $n$, the parameters for the polynomials in Table~\ref{regulars}. 

\begin{landscape}

\footnotesize
\renewcommand{\arraystretch}{2.6}
\begin{center}
\begin{longtable}{lllc}
\caption{Beta Expansions for Regular Pisot Numbers} \label{expansions} \\

\hline

 \hline
  \multicolumn{1}{l}{\textbf{Polynomial}} & \multicolumn{1}{l}{\textbf{Expansion}} & \multicolumn{1}{l}{\textbf{Restriction}} & \multicolumn{1}{c}{\textbf{Pseudo Co-Factor}}\\ \hline 
\endfirsthead

\multicolumn{4}{c}%
{{ \tablename\ \thetable{} -- continued from previous page}} \\
\hline \multicolumn{1}{l}{\textbf{Polynomial}} &
\multicolumn{1}{l}{\textbf{Expansion}} &
\multicolumn{1}{l}{\textbf{Restriction}} &
\multicolumn{1}{c}{\textbf{Pseudo Co-Factor}} \\ \hline 
\endhead

 \multicolumn{4}{r}{{Continued on next page}} \\ 
\endfoot

\endlastfoot



$\Phi_{r}(x)$&
  $\displaystyle{1^r 0^{r-1} 1}$ &
  $r \geq 1$ &
  $\displaystyle{\frac{x^r - 1}{x-1}}$  \\  
  
\hline
$\Psi_{r}(x)$&
  $\displaystyle{1^{r+1}}$ &
  $r \geq 1$ &
  $\displaystyle{1}$  \\  
  
\hline
$\mathcal{X}(x)$&
  $\displaystyle{11 (10)^\omega}$ &
  NA &
  $\displaystyle{1}$  \\  

\hline
\multirow{4}{*}{$\Phi_{A(r,n)}^-(x)$}&
  Pisot root not in $(1,2)$ &
  $1 \leq n \leq r-1$ &
  None \\*
 &
  $\displaystyle{1^r 0^{n-r} 1 0^{2r-n-2} 1 0^{n-r+1} 1^{r-1}}$ &
  $r \leq n \leq 2r - 2$ &
  $\displaystyle{\frac{x^r - 1}{x-1}}$  \\*
  
& $1^r 0^{r-2} 1 0^{r} 1^{r-1}$ &
  $n = 2r-1$ &
  $\displaystyle{\frac{x^{2r-1}-1}{(x-1)(x^r+1)}}$  \\*
  
& $1^r 0^{r-1} 1 0^{n-2r} 1 0^{r} 1^{r-1}$ &
  $n \geq 2r$ &
  $\displaystyle{\frac{x^r - 1}{x-1}}$  \\

\hline
\multirow{5}{*}{$\Phi_{A(r,2rk+j)}^+(x)$}  & 
  No Pisot root & 
  $r + (2rk+j) \leq 3$ &
  None \\* 

  &
  See Section~\ref{phiAp} &
  $1 \leq j \leq r-2$ and $k = 0$ &
  Not shown here \\*

  &
  $(1^r 0^r)^k 1^j 0 1^{r-j-2} 0 1^{j+1} 0^r (1^r 0^r )^{k-1} 1^{r-1}$ &
  $0 \leq j \leq r-2$ and $k \geq 1$ &
  $\displaystyle{\frac{x^{r(2k+1)}+1}{(x-1)(x^r +1)}}$  \\*

 &
  $(1^r 0^r)^k 1^{r-1}$ &
  $j = r-1$ &
  $\displaystyle{\frac{1}{(x-1)(x^r+1)}}$  \\* 

  &
  $(1^r 0^r)^k 1^{r-1} 0 1^{j-r} 0 (1^r0^r)^k 1^{r-1}$ &
  $r \leq j \leq 2r-1$ &
  $\displaystyle{\frac{x^{r(2k+1)}+1}{(x-1)(x^r +1)}}$  \\

\hline
\multirow{4}{*}{$\Phi_{B(r,n)}^-(x)$}  & 
  Pisot root not in $(1,2)$ & 
  $n \leq r-1$ &
  None  \\* 

  &
  $1^{2r} 0^{r-1} 1$ &
  $n = r$ &
  $\displaystyle{\frac{x^r -1}{x-1}}$ 	\\* 

  &
  $1^{r+1} 0^r 1$ &
  $n = r+1$ &
  $\displaystyle{\frac{x^{r+1} -1}{(x-1)(x^r+1)}}$   \\*
  &
  See Section~\ref{phiBm} &
  $n \geq r+2$ &
  Not shown here  \\

\hline
\multirow{3}{*}{$\Phi_{B(r,2rk+j)}^+(x)$}  & 
  $(1^r 0^r)^k 0^r 1^{r-1} 0^r (1^r 0^r)^{k-1} 1$ & 
  $j = 0$ &
  $\displaystyle{\frac{x^{r(2k+1)}+1}{(x-1)(x^r + 1)}}$  	\\* 

  &
  $(1^r 0^r)^k 1$ &
  $j = 1$ &
  $\displaystyle{\frac{1}{(x-1)(x^r + 1)}}$  	\\* 
  
  &
  See Section~\ref{phiBp} &
  $2  \leq j \leq 2r -1$ &
  Not shown here 	\\

\hline
\multirow{4}{*}{$\Phi_{C(r,n)}^-(x)$}  & 
  Pisot root not in $(1,2)$ & 
  $0 \leq n \leq r$ &
  None 	\\* 

  &
  $1^r 0^{n-r-1} 1 \left( 0^{2r-n} 1^r 0^{n-r} \right) ^\omega$ &
  $r+1 \leq n \leq 2r - 1$ &
  $\displaystyle{\frac{x^r-1}{x-1}}$ 	\\* 

 &
  $1^r 0^{r-2} 1 \left(0^{r+1} 1^{r-2} 0^{r} 10 1^r 0^{r-2} \right)^\omega$ &
  $n = 2r$ &
  $\displaystyle{\frac{(x^{2r-1} - 1)(x^{3r} + 1)}{(x - 1)(x^r + 1)}}$ \\* 

 &
  $1^r 0^{r-1} 1 \left(0^{n-2r} 1^r 0^r \right)^\omega$ &
  $n \geq 2r + 1$ &
  $\displaystyle{\frac{x^r-1}{x-1}}$ 	\\

\hline
\multirow{4}{*}{$\Phi_{C(r,2rk+j)}^+(x)$}  & 
  See Section~\ref{phiCp} & 
  $j < r$ and $k = 0$ &
  Not shown here 	\\* 
  
 &
  $(1^r 0^r)^k 0^{2rk-1} 1$ & 
  $j = 0$ and $k \geq 1$ &
  $\displaystyle{\frac{x^{2rk} - 1}{(x - 1)(x^r + 1)}}$ 	\\* 

 &
  $(1^r 0^r)^k 1^{j-1} 0 1^{r-j} 0^{r-1} 1 0^{(2rk+j)-1} 1 $ &
  $1 \leq j < r$ and $k \geq 1$, or $j = r$ &
  $\displaystyle{\frac{x^{2r(k+1)} - 1}{(x - 1)(x^r + 1)}}$ \\* 

 &
  $(1^r 0^r)^k 1^{r-1} 0 1^{j-r} (0^r 1^r)^k 0^{j-1} 1 0^{2rk+r-1} 1 $ &
  $r+1 \leq j \leq 2r-1$ &
  $\displaystyle{\frac{(x^{r(2k+1)} - 1)(x^{(2rk+j)} - 1)}{(x - 1)(x^{r} + 1)}}$ 	\\

\hline
\multirow{2}{*}{$\Psi_{A(r,n)}^-(x)$}  & 
  Pisot root not in $(1,2)$ & 
  $0 \leq n \leq r$ &
  None 	\\* 

 &
  $1^{r+1} \left( 0^{n-r-1} 1^r 0 \right)^{\omega} $ &
  $n \geq r +1$ &
  $1$ 	\\

\hline
  $\Psi_{A(r,n)}^+(x)$   & 
  See Section~\ref{psiAp} & 
  $n,r \geq 1$ &
  Not shown here  \\

\hline
\multirow{2}{*}{$\Psi_{B(r,n)}^-(x)$}  & 
  Pisot root not in $(1,2)$ & 
  $0 \leq n \leq r-1$ &
  None 	\\* 

 &
  $1^{r+1} 0^{n-r} 1^r$ &
  $n \geq r$ &
  $1$ 	\\  
  
\hline
  $\Psi_{B(r,n)}^+(x)$   & 
  See Section~\ref{psiBp} & 
  $n,r \geq 1$ &
  Not shown here 	\\

\hline
\multirow{3}{*}{$\mathcal{X}_{A(n)}^-(x)$}  & 
  Pisot root not in $(1,2)$ & 
  $1 \leq n \leq 3$ &
  None 	\\* 

 &
  $11 (10)^{k-2} 11011 ( (10)^{k-2} 0111 (01)^{k-2} 1000)^\omega$ &
  $n = 2k$ and $k \geq 2$ &
  $\displaystyle{\frac{(x^n - 1)(x^{n+1} + 1)}{x^2 - 1}}$ \\*    

 &
  $11 (10)^{k-2} 11 (00011 (10)^{k-2} 00)^\omega$ &
  $n = 2k+1$ and $k \geq 2$ &
  $\displaystyle{\frac{x^n - 1}{x^2 - 1}}$ 	\\

\hline
\multirow{5}{*}{$\mathcal{X}_{A(n)}^+(x)$}  & 
  $1001001$ &
  $n = 1$ &
  $\displaystyle{x^2 + 1}$ 	\\* 

 &
  $11$ &
  $n = 2$ &
  $\displaystyle{\frac{x^2 + 1}{x^{6} + 1}}$ \\*    

 &
  $(110)^2 10 (100)^2 1011$ &
  $n = 4$ &
  $\displaystyle{\frac{x^{11} - 1}{x - 1}}$ \\*    

 & 
  $11 (10)^{k-1} 01000 (10)^{k-1} 0 (00)^{k} 11$ &
  $n = 2k+1$ and $k \geq 1$ &
  $\displaystyle{\frac{(x^n + 1)(x^{n+1} - 1)}{x^2 - 1}}$ \\*    

 &
  $11 (10)^{k-2} 0111000 (10)^{k-3} 000010 (00)^{k-2} 11$ &
  $n = 2k$ and $k \geq 3$ &
  $\displaystyle{\frac{(x^{n-1} + 1)(x^{n+2} - 1)}{x^2 - 1}}$ 	\\

\hline
\multirow{5}{*}{$\mathcal{X}_{B(n)}^-(x)$}  & 
  Pisot root not in $(1,2)$ & 
  $1 \leq n \leq 3$ &
  None 	\\* 

 &
  $1^6 0^5 1$ &
  $n = 4$ &
  $\displaystyle{x^4 + x^2 + 1}$ \\*    

 &
  $1^4 00 1^4 0^5 1$ &
  $n = 5$ &
  $\displaystyle{x^7 + x^5 + x^2 + 1}$ \\*    

 &
  $11 (10)^{k-3} 1100000 (10)^{k-3} 001$ &
  $n = 2k$ and $k \geq 3$ &
  $\displaystyle{\frac{x^{n-2} - 1}{x^2 - 1}}$ \\*    

 &
  $11 (10)^{k-2} 1101000 (10)^{k-3} 011 (1 (00)^{k-1} 10)^\omega$ &
  $n = 2k+1$ and $k \geq 3$ &
  $\displaystyle{\frac{x^{2n-2} -x^{n-1} -x^{n-3} + 1}{x^2 - 1}}$ 	\\

\hline
\multirow{5}{*}{$\mathcal{X}_{B(n)}^+(x)$}  & 
  $10001$ &
  $n = 1$ &
  $\displaystyle{\frac{x^2 -x + 1}{x^{2} - 1}}$ 	\\* 

 &
  $101000101$ &
  $n = 2$ &
  $\displaystyle{\frac{x^5 + 1}{x^{2} - 1}}$ \\*    

 &
  $110101 (0 (1100)^2 00100 )^\omega$ &
  $n = 4$ &
  $\displaystyle{\frac{x^6(x^5+2)(x^2-x+1)-x^5-x^9}{x - 1}+x^2+1}$ \\*    

 &
  $11 (10)^{k-1} 001$ &
  $n = 2k+1$ and $k \geq 1$ &
  $\displaystyle{\frac{1}{x^2 - 1}}$ \\*    

 &
  $11 (10)^{k-2} 0101  (1 (10)^{k-3} (011)^2 (10)^{k-3} 010^4 100)^\omega$ &
  $n = 2k$ and $k \geq 3$ &
  $\displaystyle{\frac{x^4 (x^{n-1} + 1)(x^{2} -x+ 1)+x^{2n+4}-1}{x^2 - 1}}$ 	\\
  
\hline

\hline

\end{longtable}
\end{center}
\end{landscape}

The proof of correctness for each of the expansions in the above table are all similar. It should therefore be sufficient to show the proof for one of the expansions, as the rest can be verified using nearly identical arguments. 

We begin by proving the following useful lemma, which can be easier to apply than Theorem~\ref{greed} when checking that an expansion is greedy.

\begin{lemma}
\label{greedyLemma}
Let $a_1 \cdots a_n$ be a sequence in $\{0,1\}^\mathbb{N}$, and let $X = a_1 \cdots a_{i-1}$ be a leading proper subsequence, where $a_i = 0$ and $a_j =1$ for some $i < j \leq n$. Suppose that $X \geq_{\text{lex}} a_k \cdots a_{i-1} 1$ for every $k$ such that $2 \leq k \leq i-1$, and that $X \geq_{\text{lex}} a_k \cdots a_n$ for every $k$ such that $i \leq k \leq n$. Then $a_1 \cdots a_n$ represents the greedy expansion of $1$ using some base $\beta$.
\end{lemma}
\begin{proof}

It is sufficient to show that the sequence $a_1 \cdots a_n$ satisfies the condition of Theorem~\ref{greed}; that is, we need to show that $a_1 \cdots a_n >_{\text{lex}} a_k \cdots a_n$ for all $2 \leq k \leq n$. 

Indeed, if $2 \leq k \leq i -1$, then
\begin{eqnarray*}
a_1 \cdots a_n &=& X a_i \cdots a_n \\
& >_{\text{lex}} & X \\
& \geq_{\text{lex}} & a_k \cdots a_{i-1} 1 \\
& >_{\text{lex}} & a_k \cdots a_{i-1} 0 a_{i+1} \cdots a_n \\
& = & a_k \cdots a_n 
\end{eqnarray*}
and if $i \leq k \leq n$, then
\begin{eqnarray*}
a_1 \cdots a_n &=& X a_i \cdots a_n \\
& >_{\text{lex}} & X \\
& \geq_{\text{lex}} & a_k \cdots a_n 
\end{eqnarray*}
This proves the lemma.
\end{proof} 

We are now ready to prove the correctness of the expansions presented in Table~\ref{expansions}. We will focus on one of the expansions based on the root of the polynomial $\Phi_{A(r,n)}^+(x)$.

It is possible to write $n$ as $n=(2r)k+j$ for some unique choice of $k\geq 0$ and $0 \leq j \leq 2r-1$. Consider the case where $r \leq j \leq 2r-1$. As seen in Table~\ref{expansions}, the beta expansion for the Pisot root of this polynomial is
\[(1^r 0^r)^k 1^{r-1} 0 1^{j-r} 0 (1^r 0^r)^k 1^{r-1} \]

\begin{proof}
The fact that the expansion is greedy follows directly from Lemma~\ref{greedyLemma}, using the prefix $X = \left(1^r 0^r \right)^k 1^{r-1}$.

It remains to show that the expansion applies to the specified class of Pisot numbers. The expansion is valid for all roots of the corresponding companion polynomial $R(x)$; therefore, it is enough to show that $\Phi_{A(r,n)}(x)$ divides $R(x)$, as this guarantees that the Pisot root of $\Phi_{A(r,n)}(x)$ is also a root of the companion polynomial.

The companion polynomial for the expansion, $R(x)$, is computed as follows. First, an expression for the length $L$ of the expansion is determined. Next, the polynomial is written as
\[ R(x) = x^L - \sum_{a_i = 1}{x^{L-i}} \]
This expression can be simplified using a computer algebra system, such as Maple.

In the case of the expansion $(1^r 0^r)^k 1^{r-1} 0 1^{j-r} 0 (1^r 0^r)^k 1^{r-1}$, the length can be seen to be $L=4rk+r+j$. The companion polynomial is then
\begin{eqnarray*}
R(x) &=& x^{L}- \sum_{i=1}^k{\sum_{m=1}^r{x^{L-(2(i-1)r+m)}}} - \sum_{i=1}^{r-1} x^{L-(2kr+i)} - \sum_{i=1}^{j-r}x^{L-(2(k+1)r+i)} \\
&&- \sum_{i=1}^{k} \sum_{m=1}^{r}x^{L-(2(k+i-1)r+j+1+m)} - \sum_{i=1}^{r-1}x^{L-(4kr+j+1+i)}
\end{eqnarray*}
This expression simplifies to
\begin{eqnarray*}
R(x) &=& \frac{1}{\left( {x}^{2r}-1 \right)\left( x-1 \right) }\Big({x}^{4rk+r+j}-{x}^{4rk+r+j+1}-2{x}^{4rk+3r+j}+{x}^{4rk+3r+j+1} \\
&& {}-2{x}^{2r+2rk+j}+{x}^{2r+4rk+j}-{x}^{2rk+j+1} +{x}^{2rk+j+1+2r}+{x}^{2rk+r+j}-{x}^{2rk+r} \\
&&{}+{x}^{3r+2rk)}+{x}^{2rk+j}-{x}^{2r-1}+{x}^{2r+2rk-1}-{x}^{2rk+3r-1}+{x}^{2r}+{x}^{r-1}-1\Big)
\end{eqnarray*}
which then factors as 
\begin{eqnarray*}
R(x) &=& \frac{ \left( {x}^{2rk+j+r+1} -2{x}^{2rk+j+r} +{x}^{2rk+j+1} -{x}^{2rk+j} +{x}^{r} -{x}^{r-1} +1 \right) \left(x^{r(2k+1)}+1\right)} {(x-1)(x^r +1)} \\
&=&\Phi_{A(r,2rk +j)}^+(x) \cdot \left(\frac{x^{r(2k+1)}+1}{(x-1)(x^r +1)}\right)
\end{eqnarray*}
where $2rk+j = n$. Note that $x-1$ is always a cyclotomic factor of $\Phi_{A(r,n)}^+(x)$, and that
\[ \frac{x^{r(2k+1)}+1}{x^r +1} = 1 - x^r +x^{2r} - x^{3r} + \cdots + x^{2kr} \]
is a polynomial; thus the minimal polynomial of the Pisot root of $\Phi_{A(r,n)}^+(x)$ divides $R(x)$, as desired. The pseudo co-factor for the expansion is 
\[ \frac{x^{r(2k+1)}+1}{(x-1)(x^r +1)} \]
as stated in Table~\ref{expansions}.

We have thus shown that the expansion $(1^r 0^r)^k 1^{r-1} 0 1^{j-r} 0 (1^r 0^r)^k 1^{r-1}$ is a greedy expansion of 1 using the Pisot root of $\Phi_{A(r,n)}^+(x)$ as the base. The method above did not place any restrictions on the values of $n$ and $r$ in order for the stated expansion to be valid. However, we see from the expansion itself that the expansion can only apply when $j \geq r$: there is a block of $j - r$ consecutive 1s in the expansion, and the length of a block must be nonnegative. Furthermore, the initial restriction on $j$ was that $0 \leq j \leq 2r -1$. Hence the expansion can only apply when $r \leq j \leq 2r -1$, and it applies in all such cases. 

This concludes the proof of the validity of the beta expansion for the Pisot root of $\Phi_{A(r,n)}^+(x)$, where $n = 2rk+j$ and $r \leq j \leq 2r-1$. 
\end{proof}

The proofs for the remaining expansions in Table~\ref{expansions} are all similar, using the pseudo co-factors provided.

\subsection{More Complicated Expansions}
\label{complicated}

While most regular Pisot numbers have reasonably straightforward beta expansions, there are a few infinite families that have much more complex structure. The beta expansions for the Pisot roots of certain cases of $\Phi_{A(r, n)}^+(x)$, $\Phi_{B(r, n)}^-(x)$, $\Phi_{B(r, n)}^+(x)$, $\Phi_{C(r,n)}^+(x)$, $\Psi_{A(r, n)}^+(x)$, and $\Psi_{B(r, n)}^+(x)$ are examined in this section.

\subsubsection{The beta expansion of $\Phi_{A(r, n)}^+(x)$ for $1 \leq n \leq r-2$}
\label{phiAp}

The beta expansions described in Table~\ref{expansions} handle most cases for the Pisot roots of $\Phi_{A(r,n)}^+(x)$, but do not cover the case where $1 \leq n \leq r-2$. In this situation, we make use of the observation that 
\[ \Phi_{A(r, n)}^+(x) = \Phi_{A(n+1,r-1)}^+(x) \]
which can be verified algebraically. Thus when  $1 \leq n \leq r-2$, we may instead look at the beta expansion for the Pisot root of $\Phi_{A(n+1, r-1)}^+(x)$. Letting $r' = n+1$ and $n' = r-1$, we have $1 \leq r' - 2 < n$, so we can safely refer to the expansion as listed in Table~\ref{expansions}. 

\subsubsection{The beta expansion of $\Phi_{B(r,n)}^-(x)$ for $n \geq r+2$}
\label{phiBm}

When $n \geq r+2$, the general pattern of the beta expansion for the Pisot root of $\Phi_{B(r,n)}^-(x)$ is much more complicated than it is for smaller values of $n$. We construct the expansion for this family of Pisot numbers in the following way.

Let 
\[ a = \left\lceil \frac{r}{n-r} \right\rceil -1 \text{ and } k = \left\lfloor \frac{a }{ 2} \right\rfloor \]
so that $a = 2k$ or $a = 2k+1$ depending on if $a$ is even or odd, respectively. \\
Let $X$ be the string
\[ X = \prod_{i=1}^k { \left( \left( 0^{n-r-1}1 \right)^{2i-1} 0^{2ri-(2i-1)n} \left(1^{n-r-1} 0\right)^{2i} 1^{(2i+1)r-2ni} \right) }\]
and let $Y$ be the string
\begin{eqnarray*}
\lefteqn{Y =\prod_{i=1}^k  \bigg( \left( 0^{n-r-1}1 \right)^{2(k+1-i)} 0^{2(k+1-i)r - (2(k+1-i) -1)n}} \\
&&  \left(1 0^{n-r-1} \right)^{2(k+1-i)-1} 1^{(2(k+1-i)-1)r-(2(k+1-i)-2)n} \bigg) 
\end{eqnarray*} 
where $\prod$ represents the concatenation of strings.

If $a=2k$ is even, then the beta expansion is
\[ 1^r X \left(0^{n-r-1} 1 \right)^a 0^{(a+1)r-an-1} 1 0^{(a+1)n-(a+2)r} 1^{(a+1)r-an-1} Y 0^r 1 \]
and the pseudo co-factor is
\[ Q_{E(r,n,a)}(x) = \frac{\left(x^{n(a+1)} +1 \right) \left(x^{r(a+2)}-1 \right)}{\left(x-1\right)\left(x^r +1 \right)\left(x^n +1 \right)}. \]

If $a=2k$ is odd, then the beta expansion is
\begin{eqnarray*}
1^r X  \left(0^{n-r-1} 1 \right)^a 0^{(a+1)r-an} \left(1^{n-r-1} 0 \right)^a   1^{(a+1)r-an-1} 0 \\
1^{(a+1)n-(a+2)r} 0^{(a+1)r-an-1}  \left(1 0^{n-r-1} \right)^a 1^{ar - (a-1)n} Y 0^r 1
\end{eqnarray*}
and the pseudo co-factor is
\[ Q_{O(r,n,a)}(x) = \frac{\left(x^{n(a+1)} -1 \right) \left(x^{r(a+2)}+1 \right)}{\left(x-1\right)\left(x^r +1 \right)\left(x^n +1 \right)}. \]

\begin{proof}
Note that the above expansions are greedy, since they start with a block of $r$ consecutive 1s, which is also the largest set of consecutive 1s in the expansion. Thus the expansions are greedy by an application of Theorem~\ref{greed}. 

In order to prove the correctness of these expansions, it is necessary to introduce new parameters: we have that $a$ is defined by the ceiling function, which is not a simple algebraic operation, and therefore cannot be computed symbolically using the methods of Section~\ref{simple}. We thus rewrite the expressions in terms of $k$ and the new variables $M$ and $b$, defined as follows:
\begin{eqnarray*}
k &=& \left\lfloor\frac{a}{2}\right\rfloor\\
M &=& n - r \\
b &=& r - a(n-r)
\end{eqnarray*}
It can then be seen that
\begin{eqnarray*}
r &=& aM + b \\
n &=& (a+1)M + b \\
a &=& 2k \text{ or } 2k+1
\end{eqnarray*}
and so we can avoid using the ceiling definition of $a$ by switching to the variables $k$, $M$, and $b$. In order to ensure that $a=\left\lceil r/(n-r) \right\rceil -1$, we should require that $0 < b \leq M$; however, it turns out that this restriction is not necessary for the calculations.

We now construct the companion polynomials $R_{E(r,n,a,k)}(x)$ and $R_{O(r,n,a,k)}(x)$ for the expansions when $a$ is even and odd, respectively. To complete the proof, we verify the algebraic identities
\begin{eqnarray*}
\lefteqn{ R_{E(2kM+b,(2k+1)M+b,2k,k)}(x) =}  \\
 & &\Phi_{B(2kM+b,(2k+1)M+b)}^-(x) \cdot Q_{E(2kM+b,(2k+1)M+b,2k)}(x)
 \end{eqnarray*}
 and
\begin{eqnarray*} 
\lefteqn{ R_{O((2k+1)M+b,(2k+2)M+b,2k+1,k)}(x) =}  \\
 & & \Phi_{B((2k+1)M+b,(2k+2)M+b)}^-(x) \cdot Q_{E((2k+1)M+b,(2k+2)M+b,2k+1)}(x)
\end{eqnarray*}
The details for this are not shown.
\end{proof}

It is worth mentioning that when $n \geq 2r$, we have that $a = k = 0 $, and so the expansion heavily degenerates to
\[ 1^r 0^{r-1} 1 0^{n-2r} 1^{r-1} 0^{r} 1 \]
with pseudo co-factor
\[ \frac{x^r -1}{x-1}. \]
However, this situation is covered by the above general expansion, so we do not need to identify it as a separate case in Table~\ref{expansions}.

\subsubsection{The beta expansion of $\Phi_{B(r,2rk+j)}^+(x)$ for $2 \leq j \leq 2r-1$}
\label{phiBp}

When constructing the beta expansions of 1 using the Pisot roots of $\Phi_{B(r,n)}^+(x)$, it is first necessary to obtain the unique $k$ and $j$ values such that $n=(2r)k+j$. The expansion is relatively complex in the case that $2 \leq j \leq 2r-1$.

In this situation, we must determine the unique values $a$ and $b$ such that $r-1 = aj+b$, where $a \geq 0$ and $0 \leq b \leq j-1$. Note that this would not have been possible if we had $j < 2$.\\
Let $X$ be the string
\[ X= \prod_{i = 0}^a \left( \left(1^{j-1}0 \right)^i 1^{r-ij} \left(0^{j-1}1\right)^i 0^{r-ij}\right)^k \]
and let $Y$ be the string
\[ Y = \prod_{i = 0}^a \left( 0^{b+ij} \left(10^{j-1} \right)^{a-i} 1^{1+b+ij} \left(01^{j-1}\right)^{a-i} 0\right)^k \]
where $\prod$ denotes string concatenation, as before.

The desired beta expansion is then
\[ X \left(1^{j-1} 0 \right)^a 1^b 0 1^{j-b-1} Y 0^{r-1} 1 \]
with pseudo co-factor
\[ \frac{\left(x^{n(a+1)}-1\right)\left(x^{r(2(a+1)k+1)}+1\right)}
{(x-1)(x^r+1)(x^n-1)}. \]

\begin{proof}
As done in Section~\ref{phiBm}, we may express all functions in terms of the parameters $a$, $b$, $j$, and $k$ in order to complete the proof; in this manner we have
\begin{eqnarray*}
r &=& aj+b+1 \\
n &=& 2(aj+b+1)k+j 
\end{eqnarray*} 
The remainder of the proof is straightforward; greediness of the expansion follows directly from Theorem~\ref{greed} if $k = 0$, and from Lemma~\ref{greedyLemma} if $k \geq 1$. When applying the lemma, we use the prefix $X = \left(1^r 0^r \right)^k 1^*$, where $1^*$ represents the maximal block of consecutive 1s available at that point in the expansion (the size of this block depends on the values of $a$, $b$, $j$, and $k$).
\end{proof}

\subsubsection{The beta expansion of $\Phi_{C(r, 2rk+j)}^+(x)$ for $j < r$ and $k = 0$}
\label{phiCp}

The beta expansion of 1 using the Pisot root of $\Phi_{C(r,n)}^+(x)$ depends on the values of $k$ and $j$, which are uniquely defined by the equation $n = (2r)k + j$ with $k \geq 0$ and $0 \leq j \leq 2r-1$. The case when $j < r$ and $k = 0$ is therefore equivalent to saying that $n < r$.

The polynomial $\Phi_{C(r,n)}^+(x)$ has the following nice symmetric property:
\[ \Phi_{C(r,n)}^+(x) = \Phi_{C(n,r)}^+(x) \]
Thus by swapping the values of $n$ and $r$, it is sufficient to describe the expansions when $r \geq n$, which was done in Table~\ref{expansions}.

\subsubsection{The beta expansion of $\Psi_{A(r, n)}^+(x)$}
\label{psiAp}

Observe that for any choice of $r, n \geq 1$, the following identity holds:
\[ \Psi_{A(r,n)}^+(x) = \frac{\Phi_{B(n,r+1)}^+(x)}{x-1} \]
Since $x-1$ is always a cyclotomic factor of $\Phi_{B(n,r+1)}^+(x)$, it is clear that the Pisot root of $\Psi_{A(r,n)}^+(x)$ is the same as the Pisot root of $\Phi_{B(n,r+1)}^+(x)$. Thus the beta expansion for the Pisot root of $\Psi_{A(r, n)}^+(x)$ can be obtained by referring to the beta expansion for the corresponding root of $\Phi_{B(n,r+1)}^+(x)$, detailed in Section~\ref{phiBp}.

It is worth mentioning here that in the general case where $r$ is fixed and $n > (r+1)/2$, the expansion takes the form 
\[ \left(1^r 0 \right)^k 1^j 0 1^{r-j} 0^{n-1} 1 \]
where $k \geq 0$ and $0 \leq j \leq r$ are the unique integers such that $n-1 = (r+1)k+j$. This expansion is straightforward to deduce from the result in Section~\ref{phiBp}.

\subsubsection{The beta expansion of $\Psi_{B(r, n)}^+(x)$}
\label{psiBp}

Let $L = \text{lcm}(r+1, n+1)-1$, so that $L+1$ is the least common multiple of $r+1$ and $n+1$. Then the greedy beta expansion of 1 using this Pisot number is
\[ a_1 a_2 a_3 \cdots a_L\]
where
\[ 
  a_i =
  \begin{cases}
    0, & \text{if } (r+1)|i \text{ or } (n+1)|i; \\
    1, & \text{otherwise.} \\
  \end{cases}
\]

\begin{proof}
The companion polynomial for this unusual expansion is obtained using the expression
\begin{eqnarray*}
  R(x) &=& x^L - \sum_{i=0}^{L-1} x^{L-i} + \sum_{a_i = 0} x^{L-i} \\
       &=& x^L - \sum_{i=0}^{L-1} x^{i} + \left( \sum_{i = 1}^{\frac{L+1}{r+1}-1} x^{L-(r+1)i} + \sum_{i = 1}^{\frac{L+1}{n+1}-1} x^{L-(n+1)i} \right)
\end{eqnarray*}
where, instead of subtracting terms for each 1 that appears in the expansion, we add terms for each 0 that occurs.

The pseudo co-factor for this expansion is then
\[ \frac{x^{L+1}-1}{(x^{n+1}-1)(x^{r+1}-1)}. \]

Since the expansion is not presented explicitly, it is slightly more difficult to confirm that the expansion satisfies the greediness property of Theorem~\ref{greed}. Note first that the polynomial $\Psi_{B(r,n)}^+(x)$ (and hence, the corresponding beta expansion) satisfies a nice symmetric property:
\[ \Psi_{B(r,n)}^+(x) = \Psi_{B(n,r)}^+(x) \]
Thus we may assume without loss of generality that $n \geq r$. In this case, it can be seen that the expansion begins with the prefix $X = \left(1^r 0 \right)^k 1^j$ where $k \geq 0$ and $0 \leq j < r+1$ are the unique integer solutions to $n = (r+1)k + j$. The condition of Lemma~\ref{greedyLemma} follows directly upon observing that every non-leading substring of the expansion is strictly lexicographically smaller than this prefix.
\end{proof}

\subsection{Pattern Visualization}
\label{graphics}

For a new and interesting perspective, we now present the beta expansions in a different format. The expansions as written in the previous sections were written as strings, and although this is a highly efficient and precise representation, it may not be completely illuminating. It may be easier to observe and identify the underlying patterns in beta expansions when they are presented graphically.

The following figures show the beta expansions for each family of regular Pisot numbers. In the images below, each row of pixels corresponds to the greedy expansion of a particular Pisot number. The top row of pixels in each image represents the beta expansion using the root of the polynomial in the caption name, where $n = 1$; the value of $n$ increases as the row count does, until the bottom row of pixels in the image represents the beta expansion corresponding to $n = 100$. The $r$ value was arbitrarily fixed to equal $20$, as the general form of the image is roughly the same for all values of $r$ (the exception is the beta expansion for $\Psi_{B(r,n)}^+(x)$, where the expansion depends on the number of divisors of $r+1$). Within each row, a green pixel represents the digit ``0'', and a black pixel represents the digit ``1''. Note that only the first 300 digits of each expansion are shown, as this is the width of each image; thus some expansions have been truncated.

\begin{figure}[H]
  \caption{Expansions using the Pisot root of $\Phi_{A(20,n)}^-(x)$ and $\Phi_{A(20,n)}^+(x)$.}
  \centering
\includegraphics[height=75px]{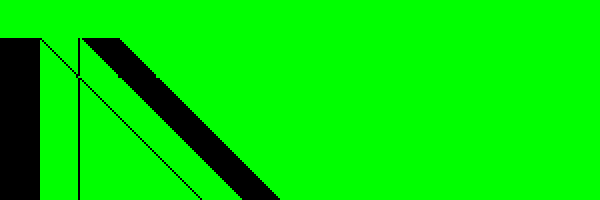}
\includegraphics[height=75px]{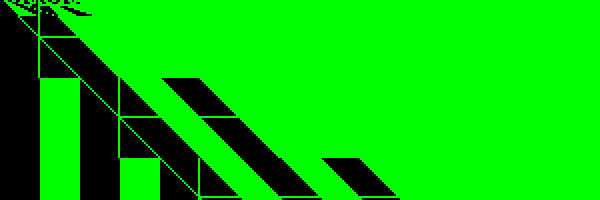}
\end{figure}
\begin{figure}[H]
  \caption{Expansions using the Pisot root of $\Phi_{B(20,n)}^-(x)$ and $\Phi_{B(20,n)}^+(x)$.}
  \centering
\includegraphics[height=75px]{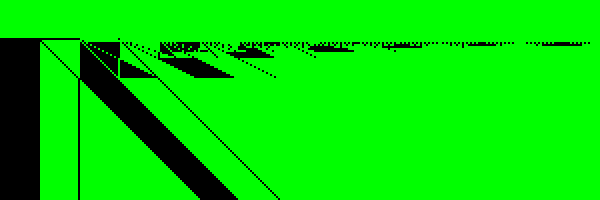}
\includegraphics[height=75px]{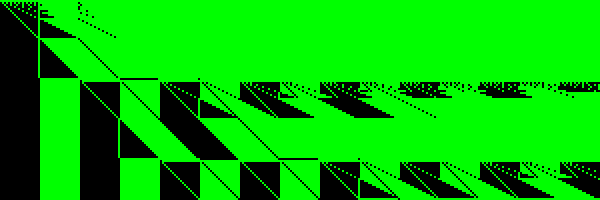}
\end{figure}
\begin{figure}[H]
  \caption{Expansions using the Pisot root of $\Phi_{C(20,n)}^-(x)$ and $\Phi_{C(20,n)}^+(x)$.}
  \centering
\includegraphics[height=75px]{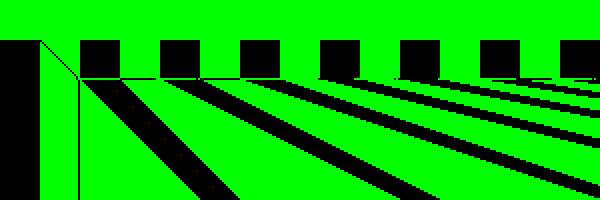}
\includegraphics[height=75px]{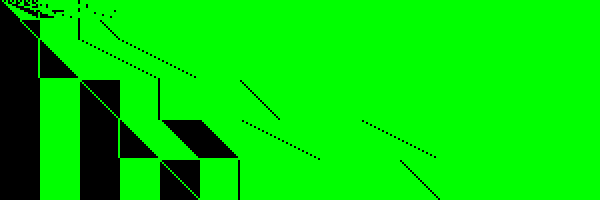}
\end{figure}
\begin{figure}[H]
  \caption{Expansions using the Pisot root of $\Psi_{A(20,n)}^-(x)$ and $\Psi_{A(20,n)}^+(x)$.}
  \centering
\includegraphics[height=75px]{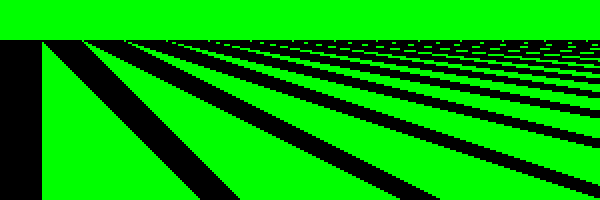}
\includegraphics[height=75px]{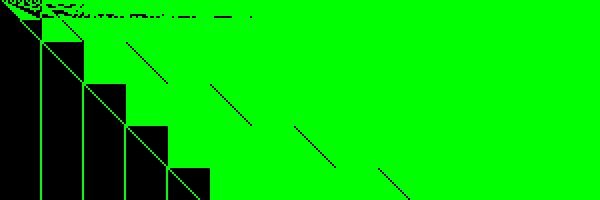}
\end{figure}
\begin{figure}[H]
  \caption{Expansions using the Pisot root of $\Psi_{B(20,n)}^-(x)$ and $\Psi_{B(20,n)}^+(x)$.}
  \centering
\includegraphics[height=75px]{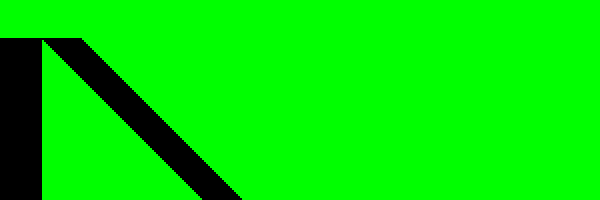}
\includegraphics[height=75px]{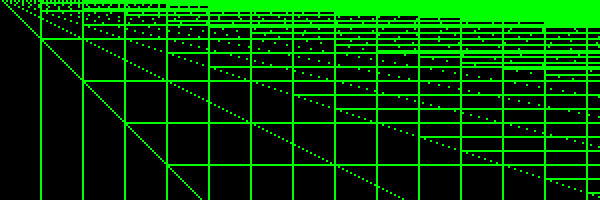}
\end{figure}
\begin{figure}[H]
  \caption{Expansions using the Pisot root of $\mathcal{X}_{A(n)}^-(x)$ and $\mathcal{X}_{A(n)}^+(x)$.}
  \centering
\includegraphics[height=75px]{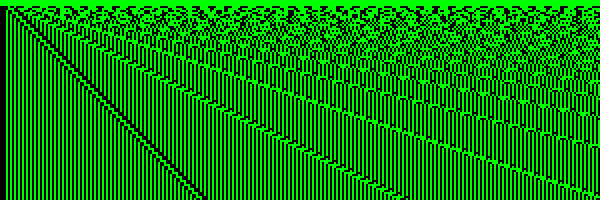}
\includegraphics[height=75px]{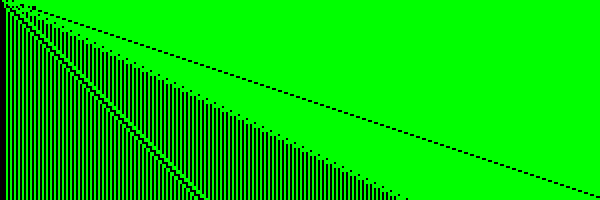}
\end{figure}
\begin{figure}[H]
  \caption{Expansions using the Pisot root of $\mathcal{X}_{B(n)}^-(x)$ and $\mathcal{X}_{B(n)}^+(x)$.}
  \centering
\includegraphics[height=75px]{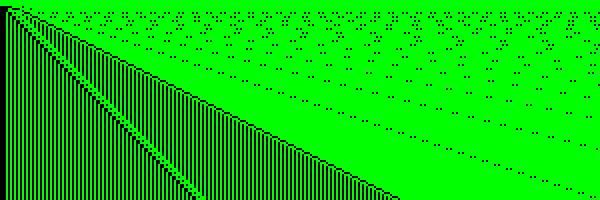}
\includegraphics[height=75px]{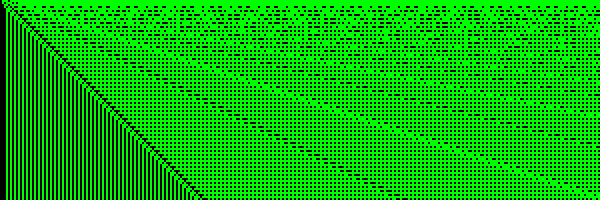}
\end{figure}
\begin{figure}[H]
  \caption{Expansions using the Pisot root of $\Phi_{n}(x)$ and $\Psi_{n}(x)$.}
  \centering
\includegraphics[height=75px]{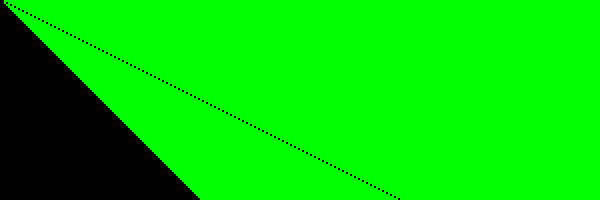}
\includegraphics[height=75px]{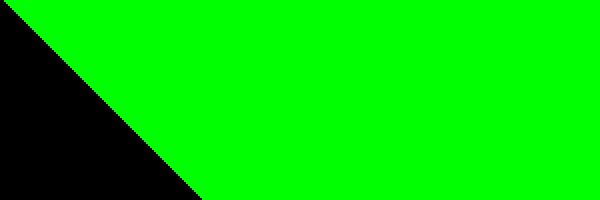}
\end{figure}
 
Note that $\mathcal{X}(x)$ is a single polynomial, rather than a sequence indexed by $n$; thus there is only one unique expansion produced in this case, which is not shown here.

\section{Consequences and Open Questions}

Now that the complete classification of regular Pisot numbers less than 2 is available, we can answer some conjectures given in earlier works.

Boyd \cite{MR1325863} had shown that for the regular Pisot numbers approaching $\phi_r$ and $\psi_r$ with $r \leq 4$, the co-factor would always be a product of cyclotomic polynomials; he conjectured that this would be true for all $r$. We can now see from Table~\ref{expansions} that the pseudo co-factors corresponding to $\Phi$ and $\Psi$ polynomials are always ratios of cyclotomic polynomials, implying that the original co-factors are indeed products of cyclotomic polynomials, as conjectured. We can also verify that non-cyclotomic co-factors only appear for the Pisot roots of $\mathcal{X}_{B(n)}^-(x)$ when $n \geq 5$ is odd, and for the Pisot roots of $\mathcal{X}_{B(n)}^+(x)$ when $n \geq 4$ is even. These non-cyclotomic polynomials are also non-reciprocal, as Boyd predicted.

Hare and Tweedle \cite{MR2444222} did a study on regular Pisot numbers satisfying the FRG property, which is defined as follows:
\begin{definition}
A Pisot number $\beta$ has a \emph{finite reversibly greedy (FRG)} beta expansion if
\begin{enumerate}
\item $d_\beta(1) = a_1 a_2 \cdots a_k$ is finite, and
\item $a_1 a_2 \cdots a_k >_{lex} a_{k-i} a_{k-i-1} a_{k-i-2} \cdots a_2$ for all $i$ with $0 \leq i \leq k-2$.
\end{enumerate}
\end{definition}

It was shown that beta expansions for certain algebraic integers could be determined if Pisot numbers had FRG beta expansions. A computational search had been done for Pisot numbers of small degree with expansions satisfying the FRG property \cite{MR2444222}, and several infinite families of numbers with the property were found. Using the results of this paper, many more classes of Pisot numbers with the FRG property can now be identified. 

We can see the following from the results in Section~\ref{results}:
\begin{itemize}
\item The expansions for the Pisot roots of $\Phi_{A(r,n)}^-(x)$, $\Phi_{C(r,n)}^+(x)$, $\Psi_{B(r,n)}^-(x)$, $\Psi_{B(r,n)}^+(x)$, $\mathcal{X}_{A(n)}^+(x)$, $\Phi_r(x)$, and $\Psi_r(x)$ always have the FRG property. The first four of these were guessed, but not proven, by Hare and Tweedle \cite{MR2444222}.
\item The expansion for the Pisot root of $\Phi_{A(r, 2rk+j)}^+(x)$ has the FRG property when $0 \leq j \leq r-1$.
\item The expansion for the Pisot root of $\Phi_{B(r,n)}^-(x)$ has the FRG property when $n = r$ or $n = r + 1$, or when $n \geq r+2$ and $a = \lceil r/(n-r) \rceil -1$ is even.
\item The expansion for the Pisot root of $\Phi_{B(r, 2rk+j)}^+(x)$ has the FRG property when $j = 0$ or $j = 1$. 
\item The expansion for the Pisot root of $\mathcal{X}_{B(n)}^-(x)$ has the FRG property when $n = 5$ or $n \geq 4$ is even.
\item The expansion for the Pisot root of $\mathcal{X}_{B(n)}^+(x)$ has the FRG property when $n = 2$ or $n \geq 1$ is odd.
\item The expansion for the Pisot root of $\Phi_{A(r, 2rk+j)}^+(x)$ is finite, but does not have the FRG property, when $r \leq j \leq 2r -1$.
\item The expansion for the Pisot root of $\Phi_{B(r,n)}^-(x)$ is finite, but does not have the FRG property, when $n \geq 2$ and $a = \lceil r/(n-r) \rceil -1$ is odd.
\item The expansion for the Pisot root of $\Phi_{B(r,2rk+j)}^+(x)$ is finite, but does not have the FRG property, when $2 \leq j \leq 2r -1$.
\item The expansions for the remaining regular Pisot numbers in the interval $(1,2)$ are not finite, and therefore do not have the FRG property.
\end{itemize}

An interesting question is to determine, given $d \geq 1$, the fraction of regular Pisot numbers of degree $\leq d$ that have expansions satisfying the FRG property. Although we know which expansions have the property, the answer to this question is not immediately apparent. This is because knowing the values of $n$ and $r$ does not give us the degree of the minimal polynomial for a regular Pisot number, and also because of the complicated indexing system used to organize the expansions in Table~\ref{expansions}.

It is an open problem to determine the beta expansions for the remaining Pisot numbers. We did not look at the beta expansions of the irregular Pisot numbers in the interval $(1,2)$, nor did we consider the Pisot numbers larger than 2. Another possible approach would be to consider a non-greedy expansion of 1 using the Pisot number as the base; in particular, there should be some classifiable patterns when looking at the lazy expansions (which are the lexicographically smallest beta expansions).

There are also some unexplained mysteries concerning the beta expansions described in this paper. For instance, the beta expansion for Pisot numbers approaching $\phi_r$ from above frequently depends on the value of the remainder when $n$ is divided by $2r$. Similarly, the expansion for Pisot numbers approaching $\psi_r$ from above also depends on the value of the remainder when $n$ is divided by $r+1$. It is unclear why a relationship exists between the modulus pattern of the expansions and the length of the beta expansion in the limit (as the greedy expansions of $\phi_r$ and $\psi_r$ have lengths $2r$ and $r+1$, respectively). There are several other curiosities that appear in the patterns of beta expansions, which although carefully observed, do not have a satisfactory explanation as of yet. It would be interesting to find some justification as to why some of the patterns and relationships appear, as this could perhaps be used to obtain more general results, such as the greedy beta expansions of 1 using other numbers as the base.

\section*{Acknowledgements}
I would like to thank my supervisor, Dr.\ Kevin Hare, for his helpful support and guidance.

\bibliographystyle{abbrv}
\bibliography{Sources}

\bigskip
\hrule
\bigskip

\noindent 2000 {\it Mathematics Subject Classification}:
Primary 11A63; Secondary 11B83.

\noindent \emph{Keywords: } Pisot number, beta expansion.

\end{document}